\documentclass[12pt,reqno]{amsart}
\usepackage{amsmath, amssymb, amsthm}

\renewcommand{\(}{\left\(}
\renewcommand{\)}{\right\)}
\renewcommand{\[}{\left\[}
\renewcommand{\]}{\right\]}

\numberwithin{equation}{section}
 \theoremstyle{plain}
\newtheorem{theorem}{Theorem}[section]

\newtheorem{remark}[]{Remark}

\newtheorem{conjecture}[theorem]{Conjecture}

\newtheorem{corollary}[theorem]{Corollary}

   \makeatletter
\def\proof{\@ifnextchar[{\@oproof}{\@nproof}}
\def\@oproof[#1][#2]{\trivlist\item[\hskip\labelsep\textit{#2 Proof of\
#1.}~]\ignorespaces}
\def\@nproof{\trivlist\item[\hskip\labelsep\textit{Proof.}~]\ignorespaces}

\makeatother
\begin{document}
\title[On the number of missing integers in partitions]{On the number of missing integers in partitions}
\author{Subhash Chand Bhoria}
\address{Subhash Chand Bhoria, Department of Mathematics, Pt. Chiranji Lal Sharma Government College, Urban Estate, Sector-14, Karnal, Haryana - 132001, India.}
\email{scbhoria89@gmail.com}

\author{Pramod Eyyunni}
\address{Pramod Eyyunni, Department of Mathematics,
Birla Institute of Technology and Science Pilani, Vidya vihar, Pilani, Rajasthan - 333031, India.} 
\email{pramod.eyyunni@pilani.bits-pilani.ac.in}
%

\author{Subhrangsu Santra}
\address{Subhrangsu Santra, Department of Mathematics,
Birla Institute of Technology and Science Pilani, Vidya vihar, Pilani, Rajasthan - 333031, India.} 
\email{p20250069@pilani.bits-pilani.ac.in, subhasantra333@gmail.com}

\thanks{$2020$ \textit{Mathematics Subject Classification.} Primary 11P81, 11P83, 11P84; Secondary 05A17. }

\begin{abstract}
In the preceding decade, Andrews and Newman resurrected the concept of a `minimal excludant' of a partition ($mex$, for short), namely, the least positive missing integer in a partition. Subsequently, several authors have not only studied its generalizations, analogues and the like but also connected the mex to several important partition statistics. In the present paper, we study the set of missing positive integers as a whole, in two different classes of partitions, namely, unrestricted partitions and overpartitions. To be precise, a $missing \ integer$ is a positive integer that is less than the largest part of a partition and which does not occur as a part. In particular, we examine the number of partitions with a given number of missing integers, determine congruences for two pairs of functions associated to them, and propose three bias type inequality conjectures for these functions.
\end{abstract}  
\maketitle
\section{Introduction}
Many statistics have been attached to integer partitions of a positive integer such as the number of parts, largest part, rank, crank, minimal excludant ($mex$) and so on. We pay particular attention to the last named of the above list. Grabner and Knopfmacher \cite{grabner} introduced the mex under the name `least gap' of a partition, namely, the smallest positive integer that is not a part of the partition. They mainly studied its analytic properties, but the concept really took off in the second half of the last decade due to two works of Andrews and Newman \cite{andrewsnewmanI, andrewsnewmanII}. They examined the mex from a combinatorial standpoint and connected it to the rank and crank statistics. This inspired many mathematicians to work on the generalizations, analogues etc. of the mex and its relation to more partition statistics. 

In this work, we explore a relatively new partition statistic closely related to mex, that is, the number of missing integers in a partition, which are greater than or equal to $j$. By $missing \ integers$ in a partition, we mean those positive integers that are less than the largest part of a partition and which do not occur as a part. For example, we can see that the missing integers greater than or equal to $1$ in the partition $9+9+7+5+2$ of $32$ are $1, 3, 4, 6$ and $8$, totalling five in number. So the mex is the smallest missing integer of a partition except for the case of gap-free partitions, partitions in which every positive integer less than the largest part appears as a part, where there are no missing integers but the mex is still defined (one more than the largest part). Missing integers of a partition recently appeared in the work of Campioni \cite{Campioni}, in the context of studying structural connections between unrefinable partitions into distinct parts and numerical semigroups.

A part of our motivation comes from the generalized mex statistic $mex_j(\pi)$ of a partition, introduced by Hopkins, Sellers, and Stanton \cite{HoSeStan2022}. The number $mex_j(\pi)$ is the least positive integer greater than or equal to $j$ missing from the partition $\pi$. As mentioned before, there is now a rich collection of results on mex, studied by many mathematicians. The interested reader can refer to \cite{andrewsnewmanI, andrewsnewmanII, ballantine-merca, barman-singh, BBEM2023, KBEM2021, BEL22, BEM22, Chakraborty-Ray, sellers-dasilva} for more information. On similar lines, Chern \cite{chernmaex} examined the maximal excludant ($maex$) of a partition, the largest integer less than the largest part that does not occur as a part of the partition. As we can see, both mex and maex come under the category of missing integers (except for gap-free partitions).

The meat of the paper is divided into two principal sections. The first one deals with missing integers in unrestricted partitions, while also extending this concept to integers with multiplicity less than $k$. As regards the second section, it contains investigations on missing integers in overpartitions analogous to those in the first section. We end with a concluding section on possible directions of further inquiry.

\section{Missing Integers in Unrestricted Partitions}
In the present section, we attempt to count the number of partitions of a number $n$ in which exactly $m$ positive integers less the largest part but greater than or equal to $j$ are missing, i.e., partitions of $n$ with $m$ missing integers $\geq j$. We denote this number by $P_j(n,m)$, where $j \in \mathbb{N}$. For example, with $n=5, \ m=1$ and $j=1$, we have $P_1(5,1)=2$ as $3+2$ and $3+1+1$ are the concerned partitions.   

Throughout the paper we shall assume $|q|<1$ and adopt the following $q$-product notation: For $a \in \mathbb{C}, \ n \in \mathbb{N}$, we define
\begin{align*}
(a;q)_n &=(1-a)(1-aq)\cdots(1-aq^{n-1}); \\
(a;q)_{\infty}&=\prod_{j=1}^{\infty} (1-aq^{j-1}).
\end{align*}
\subsection{Main Results}
\begin{theorem} \label{P_j(n,m)GFN}
For $w \in \mathbb{C}$, we have
 \begin{align}
\sum_{n=1}^{\infty}\sum_{m=0}^{\infty} P_j(n,m)w^mq^n=\frac{w^{-j}}{(\frac{w-1}{w};q)_j}\left(\frac{((w-1)q;q)_{\infty}}{(wq;q)_{\infty}}-\sum_{\ell=0}^{j-1}\frac{(\frac{w-1}{w};q)_{\ell}}{(q;q)_\ell}(wq)^\ell\right).\label{W_GFN_U}
\end{align}   
\end{theorem}
A special case of the above identity, namely for $j=1$, gives rise to a closed product form on the right hand side as given below. Denoting $P_1(n,m)$ by $P(n,m)$ and assuming $P(0,0)=1$, we have
\begin{corollary}\label{P_1(n,m)GFN}
Let $P(n,m)$ be the number of partitions of $n$ in which exactly $m$ positive integers less than the largest part are missing. Then
\begin{align}
\sum_{n=0}^{\infty}\sum_{m=0}^{\infty} P(n,m)w^mq^n=\frac{\big((w-1)q;q\big)_{\infty}}{(wq;q)_{\infty}}.\label{Main_GFN}
\end{align}
\end{corollary}
Setting $w=0$ gives the following well-known identity.
\begin{corollary} Let $P(n,0)$ be the number of partitions of $n$ in which no integer less than the largest part is missing, that is, gap free partitions, usually denoted by $\mathcal{P}^*(n)$. We thus obtain
\begin{align*}
\sum_{n=0}^{\infty} \mathcal{P}^*(n)q^n=(-q;q)_{\infty}.
\end{align*}
\end{corollary}
Performing differentiation with respect to $w$ in \eqref{Main_GFN} and putting $w=0$ yields the generating function for the coefficients corresponding to $w^1$, that is, $P(n,1)$.
\begin{corollary}\label{P_1(n,1)GFN}
We have the following identity:
\begin{align*}
\sum_{n=0}^{\infty}P(n,1)q^n=(-q;q)_{\infty}\sum_{k=1}^{\infty}\frac{q^{2k}}{1+q^k}.
\end{align*}
Since $P(n,1)$ is the number of partitions of $n$ in which exactly one integer is missing, this can also be interpreted as the number of partitions of $n$ whose \emph{mex} and \emph{maex} coincide. Furthermore, from the right hand side above, this quantity is equinumerous with the number of partitions of $n$ in which exactly one part repeats twice, while all other parts occur only once, see OEIS[A090858] \cite{OEIS}.
\end{corollary}
For instance, with $n=7$, partitions with one integer missing are $4+2+1, \ 3+3+1, \ 3+2+2, \ 3+1+1+1+1$, and the partitions of the second type as stated in the corollary are $5+1+1, \ 3+3+1, \ 3+2+2, \ 3+2+1+1$, giving us four of each kind.

Next, through the coefficients in \eqref{Main_GFN}, one might also connect the maximal excludant with the second minimal excludant ($second \ mex$, for short, which is the second smallest positive integer that is not a part of the partition) studied by Kaur, Rana and the second author \cite{KRE24}. Since $P(n,2)$ is the number of partitions of $n$ in which exactly two integers are missing, this is the same as the number of partitions of $n$ in which second mex and maex coincide. One might, in the spirit of the proof of Corollary \ref{P_1(n,1)GFN}, differentiate \eqref{Main_GFN} twice with respect to $w$ and then set $w=0$ to obtain the generating function of $P(n, 2)$. We leave the details to the reader.

We now define the following two functions: 
\begin{align*}
\mathcal{M}_e(n) &=\sum_{k=0}^{\lfloor \frac{n-1}{2} \rfloor}P(n, 2k),\\
\mathcal{M}_o(n) &=\sum_{k=1}^{\lfloor \frac{n}{2} \rfloor}P(n, 2k-1).
\end{align*}
The two functions $\mathcal{M}_e(n)$ and $\mathcal{M}_o(n)$ denote the number of partitions of $n$ with an even/odd number of missing integers respectively. The respective upper bounds in the two sums above attest to the fact that the maximum number of missing integers in a partition of $n$ is $n-1$, corresponding to the partition $n$. By letting $w=1$ in \eqref{Main_GFN}, we obtain the generating function of $p(n)$, that is $\frac{1}{(q;q)_{\infty}}$, as expected. We now replace $w$ with $-1$ instead, giving us an intriguing series, which we believe is not combinatorially well-known.
\begin{corollary}
We see that
\begin{align*} 
\sum_{n=0}^{\infty}\sum_{m=0}^{\infty} P(n,m)(-1)^mq^n=\sum_{n=0}^{\infty}\left(\mathcal{M}_e(n)-\mathcal{M}_o(n)\right)q^n=\frac{\big(-2q;q\big)_{\infty}}{(-q;q)_{\infty}}. 
\end{align*}
\end{corollary}
The $q$-product on the right hand side above yields the generating function for the number of partitions of $n$ with an even number of missing integers minus those with an odd number of missing integers. This sequence can be found in OEIS [A268498] \cite{OEIS}.

Before proceeding onto some conjectures, we describe a congruence result modulo $3$.
 \begin{theorem}\label{P_1(n,m)(w=-1)GFN}
 For $n\in\mathbb{N},$ we have
\begin{align*}
\mathcal{M}_e(n)-\mathcal{M}_o(n)\equiv
    \begin{cases}
            0 \quad(\text{mod 3}), &         \text{if n is not a square},\\
            1\quad(\text{mod 3}), &         \text{if n is an odd square},\\
            2\quad(\text{mod 3}), &         \text{if n is an even square.}
    \end{cases}
\end{align*}
\end{theorem}
 In their recent work, Kim, Kim and Lovejoy \cite{KKLparitybias} discovered a remarkable bias in the parity of partitions. They proved that the number of partitions of $n$ with more odd parts than even parts $P_o(n)$ is greater than the number of partitions of $n$ with more even parts than odd parts $P_e(n)$, the only exception being $n=2$.

In the same spirit, we propose the following conjecture regarding a bias in the parity of number of missing integers.
\begin{conjecture} For $n>34$, we have
\begin{align*}
\mathcal{M}_e(n)>\mathcal{M}_o(n).
\end{align*}
\end{conjecture}
In other words, the number of partitions of $n$ with an even number of missing integers is always greater than the number of partitions of $n$ with an odd number of missing integers for $n>34$. Furthermore, we observe another conjecture with a reversed inequality for $n > 10$.
\begin{conjecture} Let $n>10$. Then 
\begin{align*}
\mathcal{M}_e(n)-\mathcal{M}_o(n)<q(n),
\end{align*} 
where $q(n)$ is the number of partitions of $n$ into distinct parts.
\end{conjecture}

\subsection{Preliminaries}
One of the crucial tools in $q$-series and partitions is the $q$-binomial theorem. For $|q|<1$ and $ a,z \in \mathbb{C}$, one has
\begin{align}
\sum_{n=0}^{\infty} \frac{(a;q)_n}{(q;q)_n} z^n = \frac{(az;q)_{\infty}}{(z;q)_{\infty}}\label{q-binomial}.
\end{align}
Also, recall the very useful Jacobi triple product identity:
\begin{align*}
\sum_{n=-\infty}^{\infty} z^nq^{n^2} =(-q/z;q)_{\infty}(-zq;q)_{\infty}(q^2;q^2)_{\infty}. 
\end{align*}
At $z=-1$, this becomes
\begin{align}
\sum_{n=-\infty}^{\infty} (-1)^nq^{n^2} = 1+2\sum_{n=1}^{\infty} (-1)^nq^{n^2} =(q;q)^2_{\infty}(q^2;q^2)_{\infty}\label{JTPI}.
\end{align}
We also call attention to the famous partition identity due to Euler, namely,
\begin{align}
(-q;q)_{\infty}=\frac{1}{(q;q^2)_{\infty}}.\label{euler}
\end{align}
\subsection{Proofs of the main results}
\begin{proof}[Theorem \textup{\ref{P_j(n,m)GFN}}][]
We begin by writing the product: 
\begin{align*}
\frac{1}{(q;q)_{j-1}}(w+q^j+q^{2j}+\cdots)\times \cdots \times (w+q^{(\ell-1)}+q^{2(\ell-1)}+\cdots)\times(q^{\ell}+q^{2\ell}+\cdots).
\end{align*}
Observe that $w$ replaces `1' in the standard expressions for the generating functions of partitions and thus keeps track of all parts which occur zero times, i.e., those that are missing. More precisely, the exponent of $w$ counts the number of missing integers $\geq j$ and less than the largest part $\ell$ of the partition. A closed form of the above product may be written as follows:
\begin{align*}
&\frac{1}{(q;q)_{j-1}} \times \left(w + \frac{q^j}{1 - q^j} \right) \times \cdots \times \left(w + \frac{q^{\ell - 1}}{1 - q^{\ell - 1}}\right) \times \frac{q^{\ell}}{1 - q^{\ell}} \\
& =\frac{q^{\ell}w^{\ell-j}}{(q;q)_{\ell}}\left(\frac{(w-1)q^j}{w};q\right)_{\ell-j}.
\end{align*}.
Now, multiplying and dividing by $(\frac{w-1}{w};q)_j$ and then summing over $\ell$ from $j$ to $\infty$, we arrive at 
\begin{align*}
\frac{w^{-j}}{(\frac{w-1}{w};q)_j}\sum_{\ell=j}^{\infty}\frac{(\frac{w-1}{w};q)_{\ell}}{(q;q)_{\ell}}(wq)^{\ell}.
\end{align*}
The second sum here is a tail of the sum appearing in the $q$-binomial theorem \eqref{q-binomial}, where $a=\frac{w-1}{w}$ and $z=wq$, so upon applying the $q$-binomial theorem, we obtain
\begin{align*}
\frac{w^{-j}}{(\frac{w-1}{w};q)_j}\left(\frac{((w-1)q;q)_{\infty}}{(wq;q)_{\infty}}-\sum_{\ell=0}^{j-1}\frac{(\frac{w-1}{w};q)_{\ell}}{(q;q)_\ell}(wq)^\ell\right),
\end{align*}
and the proof is complete.
\end{proof}

\begin{proof}[Corollary \textup{\ref{P_1(n,m)GFN}}][]
Put $j=1$ in \eqref{P_j(n,m)GFN} to obtain
\begin{align*}
\sum_{n=1}^{\infty}\sum_{m=0}^{\infty} P(n,m)w^mq^n &=\frac{w^{-1}}{(\frac{w-1}{w};q)_1}\left(\frac{((w-1)q;q)_{\infty}}{(wq;q)_{\infty}}-\sum_{\ell=0}^{0}\frac{(\frac{w-1}{w};q)_{\ell}}{(q;q)_\ell}(wq)^\ell\right) \\
&= \frac{w^{-1}}{1 - \frac{w-1}{w}} \left(\frac{((w-1)q;q)_{\infty}}{(wq;q)_{\infty}} - 1 \right)  \\
&= \frac{((w-1)q;q)_{\infty}}{(wq;q)_{\infty}} - 1.
\end{align*}
Adding $1 = P(0, 0)$ to both sides, we get the desired result.
\end{proof}

\begin{proof}[Corollary \textup{\ref{P_1(n,1)GFN}}][]
Differentiating the left hand side of \eqref{Main_GFN} with respect to $w$ and setting $w=0$ clearly gives $\sum_{n=0}^{\infty} P(n, 1) q^n$. Coming to the right hand side, we have
\begin{align*}
\frac{d}{dw} \frac{((w-1)q; q)_{\infty}}{(wq; q)_{\infty}} &= \frac{1}{(wq; q)_{\infty}}\frac{d}{dw}((w-1)q; q)_{\infty} + ((w-1)q; q)_{\infty}\frac{d}{dw} \frac{1}{(wq; q)_{\infty}} \\
&= \frac{1}{(wq; q)_{\infty}} \times ((w-1)q; q)_{\infty} \sum_{n=1}^{\infty}\frac{-q^n}{1 - (w-1) q^n} \\
&+ ((w-1)q; q)_{\infty} \times \frac{-1}{(wq; q)_{\infty}^2} \times (wq; q)_{\infty} \times \sum_{n=1}^{\infty} \frac{-q^n}{1 - wq^n} \\
&= -\frac{((w-1)q; q)_{\infty}}{(wq; q)_{\infty}}\sum_{n=1}^{\infty}\frac{q^n}{1 - (w-1) q^n} + \frac{((w-1)q; q)_{\infty}}{(wq; q)_{\infty}} \sum_{n=1}^{\infty} \frac{q^n}{1 - wq^n},
\end{align*}
which on setting $w=0$ becomes
\begin{align*}
-(-q; q)_{\infty}\sum_{n=1}^{\infty} \frac{q^n}{1 + q^n} + (-q; q)_{\infty}\sum_{n=1}^{\infty}q^n &=(-q; q)_{\infty}\left(\sum_{n=1}^{\infty} -\frac{q^n}{1 + q^n} + q^n \right) \\
&  = (-q; q)_{\infty}\sum_{n=1}^{\infty}\frac{q^{2n}}{1 + q^n}.
\end{align*}
As far as the combinatorial interpretation stated in the corollary is concerned, note that
\begin{equation*}
(-q; q)_{\infty}\sum_{n=1}^{\infty}\frac{q^{2n}}{1 + q^n} = \sum_{n=1}^{\infty} q^{2n} \prod_{\substack{k=1 \\ k \neq n}}^{\infty} (1+q^k).
\end{equation*}
The summand on the right hand side above, $q^{2n} \prod_{\substack{k=1 \\ k \neq n}}^{\infty} (1+q^k)$, generates the partitions where $n$ occurs exactly twice and all other parts are distinct. Hence, summing over all positive integers $n$ generates all the partitions where exactly one integer appears exactly twice and all other parts are distinct. This finishes the proof.
\end{proof}

\begin{proof}[Theorem \textup{\ref{P_1(n,m)(w=-1)GFN}}][]
We shall start with writing the generating function as
\begin{align}
\sum_{n=0}^{\infty}\left(\mathcal{M}_e(n)-\mathcal{M}_o(n)\right)q^n&=\frac{\big(-2q;q\big)_{\infty}}{(-q;q)_{\infty}}\nonumber\\&=\big(-2q;q\big)_{\infty}(q;q^2)_{\infty} \quad \text{(by \eqref{euler})}\nonumber\\
&\equiv \big(q;q\big)_{\infty}(q;q^2)_{\infty} \quad\text{(mod 3)}\nonumber\\
&\equiv \big(q;q^2\big)^2_{\infty}(q^2;q^2)_{\infty} \quad\text{(mod 3)}.\label{J-product form mod 3}
\end{align}
Now, invoking the Jacobi triple product identity \eqref{JTPI} we have
\begin{align*}
\big(q;q^2\big)^2_{\infty}(q^2;q^2)_{\infty}=\sum_{n=-\infty}^{\infty}(-1)^n q^{n^2}=1+2\sum_{n=1}^{\infty}(-1)^n q^{n^2}.
\end{align*}
Therefore, \eqref{J-product form mod 3} reshapes as
\begin{align*}
\sum_{n=0}^{\infty}\left(\mathcal{M}_e(n)-\mathcal{M}_o(n)\right)q^n&\equiv \big(q;q^2\big)^2_{\infty}(q^2;q^2)_{\infty} \quad\text{(mod 3)}\\ &= 1+2\sum_{n=1}^{\infty}(-1)^n q^{n^2}\quad\text{(mod 3)}\\ & \equiv 1+\sum_{n=1}^{\infty}(-1)^{n-1} q^{n^2}\quad\text{(mod 3)}.
\end{align*}
Now comparing the coefficients of $q^n$, for $n\geq 1$, on both sides in the above expression, we find that only the square powers survive modulo 3. Furthermore, at these powers, the values $\pm$ \emph{1(mod 3)} are taken according as $n$ is odd or even respectively. This finishes the proof.
\end{proof}

\subsection{Number of distinct positive integers less than the largest part, which have multiplicity less than $k$, in all unrestricted partitions of $n$}
Merca and Yee \cite{MercaYee} studied the sum of parts with multiplicity at least $2$ in all the partitions of $n$. We look at a similar theme and consider the parts with multiplicity less than $k$. To begin with, we consider a new partition statistic, namely, the largest part minus the number of parts with multiplicity at least $k$ (the latter denoted by $\nu_{d,\geq k}(\pi)$ for a partition $\pi$). Note that this statistic can also be interpreted as the number of distinct positive integers less than or equal to the largest part, which appear less than $k$ times. This is because every integer from 1 to the largest part either appears with multiplicity at least $k$ or multiplicity less than $k$ (this includes the missing integers as well). We denote the number of distinct positive integers in the partition $\pi$ that are less than or equal to the largest part, which appear less than $k$ times, by $\nu_{d,<k}(\pi)$. We also define
\begin{align*}
\nu_{D,<k}(n)&:= \sum_{\pi \in \mathcal{P}(n)}\nu_{d,<k}(\pi)= \sum_{\pi \in \mathcal{P}(n)}(\ell(\pi) - \nu_{d,\geq k}(\pi)) \\
& = \sum_{\pi \in \mathcal{P}(n)} \ell(\pi) - \sum_{\pi \in \mathcal{P}(n)} \nu_{d,\geq k}(\pi) =   \sigma L(n)-\nu_{D,\geq k}(n),
\end{align*}
where
\begin{align*}
\sum_{\pi \in \mathcal{P}(n)}\nu_{d,\geq k}(\pi)=\nu_{D,\geq k}(n),  \quad \text{and} \quad \sum_{\pi \in \mathcal{P}(n)}\ell(\pi)=\sigma L(n).
\end{align*}
Note that, by conjugation, $\sigma L(n) = t(n)$, the total number of parts in all the partitions of $n$. We then have the following result:
\begin{theorem} \label{Multiplicity<k}
For a positive integer $k$, we have
\begin{align}
\sum_{n=0}^{\infty}\nu_{D,<k}(n)q^n=\frac{1}{(q;q)_{\infty}}\sum_{\substack{n=1\\ n \neq k}}^{\infty}\frac{q^n}{1-q^n}.
\label{QGFNmulti<k}.
\end{align}
Hence, the number of distinct integers less than or equal to the largest part appearing less than $k$ times in all the partitions of $n$ equals the number of parts different from $k$ in all the partitions of $n$.
\end{theorem}
Setting $k=1$ in the above result gives us
\begin{corollary}
The number of missing integers in all the partitions of $n$ equals the number of parts different from $1$ in all the partitions of $n$.
\end{corollary}
As an example, consider $n=4, k=2$. We list the partitions of $4$ along with the integers less than or equal to the largest part which appear less than two times in each of them in the parentheses. We have $4 \ (1, 2, 3, 4), \ 3+1 \ (1,2,3), \ 2+2 \ (1), \ 2+1+1 \ (2), \ 1+1+1+1$ (none). So, in total there are $4+3+1+1+0 = 9$ such integers over all the partitions of 4. Now, the total number of parts different from $2$ are also $9$, because the total number of parts is $12$ and the number of $2$s is $3$. 
\begin{proof}[Theorem \textup{\ref{Multiplicity<k}}][]
We start by writing a tri-variate generating function in $z,w$ and $q$ where the exponent of $z$ keeps track of the number of parts with multiplicity at least $k$ in a partition and that of $w$ keeps track of the largest part. Therefore,
\begin{align}
L(w,z;q)&=\sum_{n=1}^{\infty}\left(\frac{1-q^k}{1-q}+\frac{zq^k}{1-q}\right) \times \left(\frac{1-q^{2k}}{1-q^2}+\frac{zq^{2k}}{1-q^2}\right) \times\cdots\nonumber\\
&\times\left(\frac{1-q^{(n-1)k}}{1-q^{n-1}}+\frac{zq^{(n-1)k}}{1-q^{(n-1)}}\right) \times\left(\frac{q^n-q^{nk}}{1-q^{n}}+\frac{zq^{nk}}{1-q^n}\right)w^n\nonumber\\
&=\sum_{n=1}^{\infty}\frac{\left((1-z)q^k;q^k\right)_{n-1}w^n}{(q;q)_n}\left(q^n-(1-z)q^{nk}\right)\label{L function}.
\end{align}
Now, put $z=w^{-1}$ in \eqref{L function} to get
\begin{align*}
L(w,w^{-1};q)&=\sum_{n=1}^{\infty}\frac{\left(\frac{w-1}{w}q^k;q^k\right)_{n-1}w^n}{(q;q)_n}\left(q^n-\frac{w-1}{w}q^{nk}\right)\label{L(w,w^-1)}.
\end{align*}
Here the exponent of $w$ is basically the largest part minus the number of distinct parts whose multiplicity is at least $k$. Therefore, if $f(m, n)$ denotes the number of partitions $\pi$ of $n$ with $\ell(\pi) - \nu_{d,\geq k}(\pi) = \nu_{d, < k}(\pi) = m$, then we have
\begin{equation}\label{diff}
\sum_{n=0}^{\infty} \sum_{m=0}^{\infty} f(m, n) w^m q^n = \sum_{n=1}^{\infty}\frac{\left(\frac{w-1}{w}q^k;q^k\right)_{n-1}w^n}{(q;q)_n}\left(q^n-\frac{w-1}{w}q^{nk}\right).
\end{equation}
Differentiating the left hand side of \eqref{diff} with respect to $w$ and putting $w=1$ gives $\sum_{n=0}^{\infty} \left(\sum_{m=0}^{\infty} m f(m, n) \right)q^n$. But observe that $\sum_{m=0}^{\infty} m f(m, n) = \sum_{\pi \in \mathcal{P}(n)}\nu_{d,<k}(\pi) =  \nu_{D,<k}(n)$. Thus, the left hand side becomes $\sum_{n=0}^{\infty}\nu_{D,<k}(n)q^n$. We now apply the same sequence of operations on the right hand side of \eqref{diff}, starting with differentiation with respect to $w$. We have
\begin{align}
& \frac{d}{dw}  \sum_{n=1}^{\infty}\frac{(\frac{w-1}{w}q^k;q^k)_{n-1}}{(q;q)_n}\left(w^n q^n- (w^n - w^{n-1})q^{nk}\right) \nonumber \\
&=   \sum_{n=1}^{\infty} \frac{(\frac{w-1}{w}q^k;q^k)_{n-1}}{(q; q)_n} \frac{d}{dw} (w^n q^n- (w^n - w^{n-1})q^{nk}) \nonumber\\
&+ \frac{(w^n q^n- (w^n - w^{n-1})q^{nk})}{(q; q)_n} \frac{d}{dw} \left(\frac{w-1}{w}q^k;q^k \right)_{n-1}. \label{diffsplit}
\end{align}
We now differentiate the two terms separately, set $w=1$, and then put everything back together. The first term yields
\begin{equation*}
 \sum_{n=1}^{\infty} \frac{(\frac{w-1}{w}q^k;q^k)_{n-1}}{(q; q)_n} (n w^{n-1} q^n- (n w^{n-1} - (n-1)w^{n-2})q^{nk}),
\end{equation*} 
which at $w=1$ transforms into
\begin{equation}\label{first term}
 \sum_{n=1}^{\infty}\frac{nq^n - q^{nk}}{(q; q)_n} =  \sum_{n=1}^{\infty} \frac{nq^n}{(q; q)_n} -  \sum_{n=1}^{\infty} \frac{(q^k)^n}{(q; q)_n} =  \sum_{n=1}^{\infty} \frac{nq^n}{(q; q)_n} - \left( \frac{1}{(q^k; q)_{\infty}} - 1\right),
\end{equation}
where we invoked \eqref{q-binomial} for the second sum in the last step above. Turning our attention to the second term in \eqref{diffsplit}, we have
\begin{align*}
 \sum_{n=1}^{\infty} \frac{(w^n q^n- (w^n - w^{n-1})q^{nk})}{(q; q)_n} \left(\frac{w-1}{w}q^k;q^k \right)_{n-1} \sum_{r=1}^{n-1}\frac{-q^{rk}/w^2}{1 - \frac{w-1}{w}q^{rk}},
\end{align*}
which for $w=1$ gives
\begin{align}
-  \sum_{n=1}^{\infty}  \frac{q^n}{(q; q)_n} \sum_{r=1}^{n-1}q^{rk} &= - \sum_{n=1}^{\infty} \frac{q^n}{(q; q)_n} \times \frac{q^k - q^{nk}}{1 - q^k} \nonumber\\
& = -\frac{q^k}{1 - q^k}\sum_{n=1}^{\infty} \frac{q^n}{(q; q)_n} + \frac{1}{1 - q^k} \sum_{n=1}^{\infty} \frac{(q^{k+1})^n}{(q; q)_n} \nonumber\\
&= -\frac{q^k}{1-q^k}\left(\frac{1}{(q;q)_{\infty}}-1\right) + \frac{1}{1-q^k}\left(\frac{1}{(q^{k+1};q)_{\infty}}-1\right), \label{second term}
\end{align}
where we again made use of \eqref{q-binomial} for both the sums in the penultimate step.
Putting \eqref{first term} and \eqref{second term} together and recalling that their sum equals $\sum_{n=0}^{\infty}\nu_{D,<k}(n)q^n$, we finally obtain
\begin{align}
    \sum_{n=0}^{\infty}\nu_{D,<k}(n)q^n &=\sum_{n=1}^{\infty}\frac{nq^n}{(q;q)_n}-\frac{q^k}{1-q^k}\left(\frac{1}{(q;q)_{\infty}}-1\right)+\frac{1}{1-q^k}\left(\frac{1}{(q^{k+1};q)_{\infty}}-1\right) \nonumber\\
&-\left(\frac{1}{(q^{k};q)_{\infty}}-1\right) \nonumber \\
&=  \sum_{n=1}^{\infty}\frac{nq^n}{(q;q)_n}-\frac{q^k}{1-q^k}\frac{1}{(q;q)_{\infty}}. \label{prefinal exp}
\end{align}
Now ,
\begin{align}
\sum_{n=1}^{\infty}\frac{nq^n}{(q;q)_n}=\frac{d}{dz} \left(\sum_{n=1}^{\infty}\frac{z^n q^n}{(q;q)_{n}}\right) \Bigg|_{z=1}& = \frac{d}{dz}\left(\frac{1}{(zq;q)_{\infty}} - 1\right)\Bigg|_{z=1}\nonumber \\
& = \frac{1}{(zq;q)_{\infty}} \sum_{n=1}^{\infty}\frac{q^n}{1 - zq^n}\Bigg|_{z=1} \nonumber \\
&= \frac{1}{(q;q)_{\infty}} \sum_{n=1}^{\infty}\frac{q^n}{1 - q^n} \label{std},
\end{align}
and hence, from \eqref{prefinal exp} and \eqref{std}, we finally obtain
\begin{align*}
   \sum_{n=0}^{\infty}\nu_{D,<k}(n)q^n =  \frac{1}{(q;q)_{\infty}}\left( \sum_{n=1}^{\infty} \frac{q^n}{1 - q^n} - \frac{q^k}{1 - q^k}\right) = \frac{1}{(q;q)_{\infty}}\sum_{\substack{n=1\\ n \neq k}}^{\infty}\frac{q^n}{1-q^n}.
\end{align*}
This completes the proof of the identity \eqref{QGFNmulti<k}. To prove the combinatorial interpretation stated in the theorem, we will look at the expression
\begin{align*}
 \frac{1}{(q;q)_{\infty}}\frac{q^k}{1 - q^k} &= \frac{q^k}{(1 - q^k)^2}\prod_{\substack{n=1\\ n \neq k}}^{\infty} \frac{1}{1 - q^n} = \sum_{m=1}^{\infty} mq^{mk}\prod_{\substack{n=1\\ n \neq k}}^{\infty} \frac{1}{1 - q^n} = \sum_{n=0}^{\infty}\left(\sum_{\pi \in \mathcal{P}(n)} \nu_{\pi}(k)\right) q^n,
\end{align*}
where $\nu_{\pi}(k)$ denotes the multiplicity of $k$ in the partition $\pi$. Thus, $ \frac{1}{(q;q)_{\infty}}\frac{q^k}{1 - q^k}$ generates the number of $k$s in all partitions of $n$, and consequently, $\frac{1}{(q;q)_{\infty}}\sum_{n=1}^{\infty} \frac{q^n}{1 - q^n}$ generates the number of parts in all partitions of $n$. In conclusion, their difference would generate the number of parts different from $k$ in all partitions of $n$.
\end{proof}
\begin{remark}
A combinatorial proof of the identity \eqref{QGFNmulti<k} is highly desirable.
\end{remark}

\section{Missing Integers in Overpartitions}
From the time Corteel and Lovejoy \cite{overp} introduced the concept of overpartitions, they have been a focal point of study in partition theory. An overpartition of $n$ is a partition of $n$ in which the first occurrence of any integer may be overlined. For example, the overpartitions of $n=3$ are: $3, \ \overline{3}, \ 2+1, \ \overline{2}+1, \ 2+\overline{1}, \ \overline{2}+\overline{1}, \ \overline{1}+1+1, \ 1+1+1$. In most situations, analogous results for overpartitions as those for partitions are highly sought after. We proceed to examine the number of missing integers in overpartitions now. For instance, the number of missing integers in the above list of overpartitions of $n=3$ are respectively, $2, 2, 0, 0, 0, 0, 0, 0$.

\subsection{Main Results}
\begin{theorem} \label{OP_j(n,m)GFN}
Let $\overline{P}_j(n,m)$ be the number of overpartitions of $n$ in which exactly $m$ positive integers less than the largest part, but greater than or equal to $j$, are missing. For $j\in \mathbb{N}$ and $w\in \mathbb{C}$ we have
 \begin{align}
\sum_{n=1}^{\infty}\sum_{m=0}^{\infty} \overline{P}_j(n,m)w^mq^n=\frac{2w^{-j}(-q;q)_{j-1}}{(\frac{w-2}{w};q)_j}\left(\frac{(w-2)q;q)_{\infty}}{(wq;q)_{\infty}}-\sum_{\ell=0}^{j-1}\frac{(\frac{w-2}{w};q)_{\ell}}{(q;q)_\ell}(wq)^\ell\right).\label{Main_GFN2}
\end{align}   
\end{theorem}
A particular instance of this identity, for $j=1$, results in a $q$-product on the right hand side. Here we put $\overline{P}_1(n,m)=\overline{P}(n,m)$ and also let $\overline{P}(0,0)=1$. Then \eqref{Main_GFN2} gives us
\begin{corollary}\label{Corollary 3.2}
Let $\overline{P}(n,m)$ be the number of overpartitions of $n$ in which exactly $m$ positive integers less than the largest part are missing. Then
\begin{align}
\sum_{n=0}^{\infty}\sum_{m=0}^{\infty} \overline{P}(n,m)w^mq^n=\frac{\big((w-2)q;q\big)_{\infty}}{(wq;q)_{\infty}}.\label{W_GFN_OP}
\end{align}
\end{corollary}

Letting $w=0$ in the above equation leads to

\begin{corollary} Suppose $\overline{P}(n,0)$ is the number of overpartitions of $n$ in which no integer less than the largest part is missing, that is, gap free overpartitions. If we denote this by $\overline{\mathcal{P}}^*(n)$, then $\overline{\mathcal{P}}^*(n)$ is given by
\begin{align*}
\sum_{n=0}^{\infty} \overline{\mathcal{P}}^*(n)q^n=(-2q;q)_{\infty}.
\end{align*}
The right hand side generates the overpartitions into distinct parts, i.e., any integer can occur at most once, either in the non-overlined or the overlined versions, but not both.
\end{corollary}
For instance, the gap free overpartitions of $5$ are $2+2+1, \ \overline{2} + 2 + 1, \ 2+2+\overline{1}, \ \overline{2} + 2 + \overline{1}, \ 2+1+1+1, \ \overline{2}+2+2+1, \ 2+1+1+\overline{1}, \ \overline{2}+1+1+\overline{1}, \ 1+1+1+1+1, \ \overline{1}+1+1+1+1$, totalling 10 in number. There are 10 overpartitions of $5$ into distinct parts, namely, $5,\ \overline{5}, \ 4+1, \ \overline{4}+1, \ 4+\overline{1}, \ \overline{4} + \overline{1}, \ 3+2, \ \overline{3}+2, \ 3+\overline{2}, \ \overline{3} + \overline{2}$. Differentiating \eqref{W_GFN_OP} with respect to $w$ and putting $w=0$ extracts the coefficients $\overline{P}(n,1)$.
\begin{corollary}\label{overponemissing}
We have
\begin{align*}
\sum_{n=0}^{\infty}\overline{P}(n,1)q^n=(-2q;q)_{\infty}\sum_{k=1}^{\infty}\frac{2q^{2k}}{1+2q^k}.
\end{align*}
Here $\overline{P}(n,1)$ is the number of overpartitions of $n$ in which exactly one integer is missing. This quantity is equinumerous with the number of overpartitions in which exactly one integer occurs twice and every other integer occurs only once, either in the non-overlined or overlined versions, but not both.
\end{corollary}
As an example, there are six overpartitions of $4$ with exactly one integer missing, namely, $3+1, \ \overline{3}+1, \ 3+\overline{1}, \ \overline{3}+\overline{1}, \ 2+2, \ \overline{2}+2$ and there are again six overpartitions in which exactly one part occurs twice, with every other part occurring only once, namely, $2+2, \ \overline{2}+2, \ \overline{2}+1+1, \ 2+1+1, \ \overline{2}+\overline{1}+1, \ 2+\overline{1}+1$.
Going back to \eqref{W_GFN_OP} and letting $w=1$, we simply obtain the generating function for overpartitions, $\frac{(-q;q)_{\infty}}{(q;q)_{\infty}}$, as expected. But substituting $w=-1$ provides a new identity.
\begin{corollary}
\begin{align*}
\sum_{n=0}^{\infty}\sum_{m=0}^{\infty} \overline{P}(n,m)(-1)^mq^n=\sum_{n=0}^{\infty}\left(\overline{\mathcal{M}}_e(n)-\overline{\mathcal{M}}_o(n)\right)q^n=\frac{\big(-3q;q\big)_{\infty}}{(-q;q)_{\infty}}. \label{A268499}
\end{align*}
\end{corollary}
The left hand side above is the generating function for the number of overpartitions with an even number of missing integers, denoted by $\overline{\mathcal{M}}_e(n)$, minus the number of overpartitions with an odd number of missing integers, denoted by $\overline{\mathcal{M}}_o(n)$. This sequence $\overline{\mathcal{M}}_e(n) - \overline{\mathcal{M}}_o(n)$ can be found in OEIS [A268499].
From the OEIS webpage, we note that the coefficients of this sequence seem positive for $n>27$. Hence we conjecture the following `bias type' result analogous to the one for unrestricted partitions.
\begin{conjecture} For any positive integer $n>27$, we have 
\begin{align*}
\overline{\mathcal{M}}_e(n)>\overline{\mathcal{M}}_o(n).
\end{align*}
In other words, the number of overpartitions of $n$ with an even number of missing integers is always more than the number of overpartitions of $n$ with an odd number of missing integers for $n>27$.
\end{conjecture}
On similar lines to the unrestricted partitions, the following congruence result holds in the case of overpartititons.
\begin{theorem} \label{OP_1(n,m)(w=-1)GFN}
For $n \in \mathbb{N}$, we have
\begin{align*}
\overline{\mathcal{M}}_e(n)-\overline{\mathcal{M}}_o(n)\equiv
    \begin{cases}
            0 \quad(\text{mod 4}), &         \text{if n is not a square},\\
            2\quad(\text{mod 4}), &         \text{if n is a square}.
    \end{cases}
\end{align*}
\end{theorem}
It may be observed that for $n \in \mathbb{N}$, $\overline{\mathcal{M}}_e(n)-\overline{\mathcal{M}}_o(n)$ is always even because the number of overpartitions of $n$, namely, $\overline{p}(n) = \overline{\mathcal{M}}_e(n) + \overline{\mathcal{M}}_o(n) = \overline{\mathcal{M}}_e(n) - \overline{\mathcal{M}}_o(n) + 2 \overline{\mathcal{M}}_o(n)$ is always even. This is because each unrestricted partition of $n$ gives rise to $2^{\nu(\pi)}$ many overpartitions ($\nu(\pi)$ being the number of distinct integers that appear as parts in $\pi$), and since $\nu(\pi) \geq 1$ for each $\pi$, we have that $\overline{p}(n) = \sum_{\pi \in \mathcal{P}(n)} 2^{\nu(\pi)}$ is an even number. The above result is stronger as it characterizes the parity of $\frac{\overline{\mathcal{M}}_e(n)-\overline{\mathcal{M}}_o(n)}{2}$.
\subsection{Proofs of the main results}
\begin{proof}[Theorem \textup{\ref{OP_j(n,m)GFN}}][]
We begin by writing the product 
\begin{align*}
\frac{(-q;q)_{j-1}}{(q;q)_{j-1}}\prod_{k=j}^{\ell-1}\left(\frac{1+q^k}{1-q^k}-1+w \right)\frac{2q^\ell}{1-q^\ell}.
\end{align*}
In the above product, we do not track the missing integers less than some fixed integer $j$ but only those which are greater than or equal to $j$, but less than the largest part $\ell$ in the overpartition. Note that the $w$ in the term $\frac{1+q^k}{1-q^k}-1+w$ encodes the situation where $k$ appears in neither the non-overlined nor overlined versions, i.e., when $k$ is missing from the overpartition. A closed form of the above product may be written as
\begin{align*}
\frac{2w^{\ell-j}q^\ell(-q;q)_{j-1}}{(q;q)_{\ell}}\left(\frac{w-2}{w}q^j;q\right)_{\ell-j}.
\end{align*}.
Now, multiplying and dividing by $(\frac{w-2}{w};q)_j$ and summing over $\ell$ from $j$ to $\infty$ we arrive at 
\begin{align*}
\frac{2w^{-j}(-q;q)_{j-1}}{(\frac{w-2}{w};q)_j}\sum_{\ell=j}^{\infty}\frac{(\frac{w-2}{w};q)_{\ell}}{(q;q)_{\ell}}(wq)^{\ell}.
\end{align*}
The sum appearing in the above expression is a tail of the sum appearing in the $q$-binomial theorem \eqref{q-binomial}, where $a=\frac{w-2}{w}$ and $z=wq$. So we may rewrite it as 
\begin{align*}
\frac{2w^{-j}(-q;q)_{j-1}}{(\frac{w-2}{w};q)_j}\left(\frac{(w-2)q;q)_{\infty}}{(wq;q)_{\infty}}-\sum_{\ell=0}^{j-1}\frac{(\frac{w-2}{w};q)_{\ell}}{(q;q)_\ell}(wq)^\ell\right),
\end{align*} 
and this completes the proof.
\end{proof}
\begin{proof}[Corollary \textup{\ref{Corollary 3.2}}][]
Set $j=1$ in \eqref{Main_GFN2} to get
\begin{align*}
\sum_{n=1}^{\infty}\sum_{m=0}^{\infty} \overline{P}(n,m)w^mq^n &=\frac{2w^{-1}}{(\frac{w-2}{w};q)_1}\left(\frac{((w-2)q;q)_{\infty}}{(wq;q)_{\infty}}-\sum_{\ell=0}^{0}\frac{(\frac{w-2}{w};q)_{\ell}}{(q;q)_\ell}(wq)^\ell\right) \\
&= \frac{2w^{-1}}{1 - \frac{w-2}{w}} \left(\frac{((w-2)q;q)_{\infty}}{(wq;q)_{\infty}} - 1 \right)  \\
&= \frac{((w-2)q;q)_{\infty}}{(wq;q)_{\infty}} - 1.
\end{align*}
Add $1 = \overline{P}(0, 0)$ on both sides to get the required result.
\end{proof}
\begin{proof}[Corollary \textup{\ref{overponemissing}}][]
Differentiate both sides of \eqref{W_GFN_OP} with respect to $w$ to obtain
\begin{align*}
\frac{d}{dw} \frac{((w-2)q; q)_{\infty}}{(wq; q)_{\infty}}&= \frac{1}{(wq; q)_{\infty}} \times ((w-2)q; q)_{\infty} \sum_{n=1}^{\infty}\frac{-q^n}{1 - (w-2) q^n} \\
&+ ((w-2)q; q)_{\infty} \times \frac{-1}{(wq; q)_{\infty}^2} \times (wq; q)_{\infty} \times \sum_{n=1}^{\infty} \frac{-q^n}{1 - wq^n} \\
&= -\frac{((w-2)q; q)_{\infty}}{(wq; q)_{\infty}}\sum_{n=1}^{\infty}\frac{q^n}{1 - (w-2) q^n} + \frac{((w-2)q; q)_{\infty}}{(wq; q)_{\infty}} \sum_{n=1}^{\infty} \frac{q^n}{1 - wq^n}.
\end{align*}
Now, we set $w=0$ and get
\begin{align*}
&-(-2q; q)_{\infty}\sum_{n=1}^{\infty} \frac{q^n}{1 + 2q^n} + (-2q; q)_{\infty}\sum_{n=1}^{\infty}q^n \\
&  = (-2q; q)_{\infty}\sum_{n=1}^{\infty}\frac{2q^{2n}}{1 + 2q^n}.
\end{align*}
For the combinatorial interpretation stated in the corollary, we observe that
\begin{equation*}
(-2q; q)_{\infty}\sum_{n=1}^{\infty}\frac{2q^{2n}}{1 + 2q^n} = \sum_{n=1}^{\infty} 2q^{2n} \prod_{\substack{k=1 \\ k \neq n}}^{\infty} (1+2q^k).
\end{equation*}
The general term on the right hand side above, $2q^{2n} \prod_{\substack{k=1 \\ k \neq n}}^{\infty} (1+2q^k)$, generates the overpartitions where $n$ occurs exactly twice (either both non-overlined or one overlined and non-overlined) and all other parts are distinct, occuring at most once either as an overlined or a non-overlined part but not both. Thus, the sum over all positive integers $n$ generates the required type of overpartitions.
\end{proof}
\begin{proof}[Theorem \textup{\ref{OP_1(n,m)(w=-1)GFN}}][]
We transform the generating function as follows:
\begin{align}
\sum_{n=0}^{\infty}\left(\overline{\mathcal{M}}_e(n)-\overline{\mathcal{M}}_o(n)\right)q^n&=\frac{\big(-3q;q\big)_{\infty}}{(-q;q)_{\infty}}\nonumber\\&=\big(-3q;q\big)_{\infty}(q;q^2)_{\infty} \quad (\text{by \eqref{euler}})\nonumber\\
&\equiv \big(q;q\big)_{\infty}(q;q^2)_{\infty} \quad\text{(mod 4)}\nonumber\\
&\equiv \big(q;q^2\big)^2_{\infty}(q^2;q^2)_{\infty} \quad\text{(mod 4)}.\label{J-product form mod 4}
\end{align}
Here, an application of \eqref{JTPI} gives us
\begin{align*}
\big(q;q^2\big)^2_{\infty}(q^2;q^2)_{\infty}=\sum_{n=-\infty}^{\infty}(-1)^n q^{n^2}=1+2\sum_{n=1}^{\infty}(-1)^n q^{n^2}.
\end{align*}
Therefore, \eqref{J-product form mod 4} becomes
\begin{align*}
\sum_{n=0}^{\infty}\left(\overline{\mathcal{M}}_e(n)-\overline{\mathcal{M}}_o(n)\right)q^n&\equiv \big(q;q^2\big)^2_{\infty}(q^2;q^2)_{\infty} \quad\text{(mod 4)}\\ & = 1+2\sum_{n=1}^{\infty}(-1)^n q^{n^2}\quad\text{(mod 4)}\\ & \equiv 1+2\sum_{n=1}^{\infty}q^{n^2}\quad\text{(mod 4)}.
\end{align*}
Now comparing the coefficients of $q^n$, for $n\geq 1$, on both sides in the above expression, we find that only square powers survive modulo 4 and have a coefficient of 2 associated with them. Thus, the result follows.
\end{proof}

\section{Concluding Remarks}
In this paper, we studied the missing integers in unrestricted partitions and overpartitions. More specifically, we gave the generating function for the number of partitions with a given number, $m$, of missing integers, namely, $P(n, m)$. As a special case, we found a partition identity for $P(n, 1)$, that is, Corollary \ref{P_1(n,1)GFN}, whose bijective proof would be interesting to get. We then considered the parity of the number of missing integers and found the generating function as well as a congruence modulo $3$ for $\mathcal{M}_e(n) - \mathcal{M}_o(n)$. A combinatorial explanation of this congruence, namely Theorem \ref{P_1(n,m)(w=-1)GFN}, would be fascinating. We also posed a few conjectural inequalities regarding the relation between $\mathcal{M}_o(n)$ and $\mathcal{M}_e(n)$. A bijective proof of \eqref{QGFNmulti<k} in Theorem \ref{Multiplicity<k}, which talks about, in more generality, about the number of missing integers in all partitions of $n$, would be highly desired. We also examined similar problems for overpartitions and combinatorial proofs, in this case too, are sought after.

We now propose some additional questions and indicate directions for possible exploration.

\noindent \textbf{Question 1.} Do there exist congruences for $\mathcal{M}_e(n)-\mathcal{M}_o(n)$ and $\overline{\mathcal{M}}_e(n)-\overline{\mathcal{M}}_o(n)$ for moduli other than $3$ and $4$?\\
\textbf{Question 2.} Determine the asymptotic formulae and the Hardy-Ramanujan-Rademacher type exact formulae for $\mathcal{M}_e(n)-\mathcal{M}_o(n)$ and $\overline{\mathcal{M}}_e(n)-\overline{\mathcal{M}}_o(n)$, that is for the $q$-products $\frac{\big(-2q;q\big)_{\infty}}{(-q;q)_{\infty}}$ and $\frac{\big(-3q;q\big)_{\infty}}{(-q;q)_{\infty}}$. \\
Although the coefficient sequences for the generating functions $\frac{\big(-2q;q\big)_{\infty}}{(-q;q)_{\infty}}$ and $\frac{\big(-3q;q\big)_{\infty}}{(-q;q)_{\infty}}$ were already recorded in the OEIS[A268498, A268499], their interpretation in terms of \emph{missing integers} in partitions is new to the best of our knowledge. Interestingly, the sequences quoted above are consecutive sequences in OEIS. \\
\textbf{Question 3.} Is there an overpartition analogue of Theorem \ref{Multiplicity<k}, namely, a concise identity for the number of integers less than or equal to the largest part that appear less than $k$ times in all overpartitions of $n$.


%
%
%

\end{document}